\documentclass[reqno]{amsart}
\usepackage{amssymb}
\usepackage{mathrsfs}
\usepackage{amsmath,amstext,amsfonts,verbatim}

\usepackage[latin1]{inputenc}

\usepackage{amsmath,amsthm,amssymb}

\usepackage{amsfonts}

\usepackage{amsmath,amssymb}
\usepackage{graphicx}
\usepackage{color}
\usepackage{indentfirst}

\newtheorem{theorem}{Theorem}[section]

\theoremstyle{definition}

\newtheorem{prop}{Proposition}[section]

\theoremstyle{remark}

\numberwithin{equation}{section}

\newcommand{\neweq}[1]{\begin{equation}\label{#1}}

\def\phi{\varphi}

\def\incep{\left\{\begin{array}{cl} }
 \def\termin{\end{array}\right. }
\def\2af{2^*_\alpha}

\begin{document}

\title[The entropy of an extended map for Abelian group actions]
{\textbf{The entropy of an extended map for Abelian group actions}}

\author{Yuan Lian}
\address{College of Mathematics and Statistics, Taiyuan Normal University, Taiyuan 030619, China}
\email{andrea@tynu.edu.cn}





\begin{abstract}

In this paper, we mainly consider on the entropy of the extended map conditional to the natural
extension of a dynamical system for an Abelian group action and we calculate the entropy is zero.

\end{abstract}

\maketitle

\section{Introduction}
In this section, we give some background on measurable dynamical systems used in our discussion and the theoretical basis of this section mainly refers to reference \cite{AL} and \cite{DZ}.

A probability space is called  Lebesgue space (see the pioneering work of Rokhlin) if it is metrically isomorphic to a probability space which is the disjoint union of an at most countable (possibly empty) set $\{x_{1},x_{2},...\}$ of points each of positive measure and the space $([0,s),\mathcal{L},\lambda)$(possibly absent), where $\mathcal{L}$
is the $\sigma-$algebra of Lebesgue measure subsets of the interval $[0,s)$ and $\lambda$ is Lebesgue measure. Here $s=1-\sum p_{n}$, $p_{n}=$measure of the point $x_{n}$.

A Polish probability space $(X, \mathcal{B}_{X},\nu)$ means that $X$ is a separable topological space whose topology is metrizable by a complete metric, $\mathcal{B}_{X}$ is the Borel $\sigma$-algebra, and $\nu$ is a Borel probability measure on $X$. A Polish system $(X, \mathcal{B}_{X},\nu)$ is a measure-preserving maps $G$ on a Polish space $(X, \mathcal{B}_{X},\nu)$. A Lebesgue system $(X, \mathcal{B}_{X},\nu, G)$ is a measure-preserving maps $G$ on a Lebesgue space  $(X, \mathcal{B}_{X},\nu)$.

By a measurable dynamical $G$-system (MDS) $(Y, \mathcal{D}, \nu,G)$ we mean a measure-preserving transformations of $(Y,\mathcal{D}, \nu)$ with $e_{G}$ acting as the identity transformation. A cover of $(Y,\mathcal{D}, \nu)$ is a family $\mathcal{W}\subseteq\mathcal{D}$ satisfying $\cup_{W\in \mathcal{W}}W=Y$; if all elements of a cover $\mathcal{W}$ are disjoint, then $\mathcal{W}$ is called a partition of $(Y,\mathcal{D}, \nu)$. Denote by $C_{Y}$ and $P_{Y}$ the set of all finite covers and finite partitions of $(Y,\mathcal{D}, \nu)$, respectively. Let $\alpha$ be a partition of $(Y,\mathcal{D},\nu)$ and $y \in Y $. Denote by $\alpha(y)$ the atom of $\alpha$ containing $y$. Let $\mathcal{W}_{1},\mathcal{W}_{2}\in C_{Y}$. If
each element of $\mathcal{W}_{1}$ is contained in some element of $\mathcal{W}_{2}$ then we say that $\mathcal{W}_{1}$ is finer than $\mathcal{W}_{2}$ (denote by
$\mathcal{W}_{1}\succeq \mathcal{W}_{2}$ or $\mathcal{W}_{2}\preceq\mathcal{W}_{1}$). The join $\mathcal{W}_{1}\vee \mathcal{W}_{2}$ of $\mathcal{W}_{1}$
and $\mathcal{W}_{2}$ is given by $$\mathcal{W}_{1}\vee \mathcal{W}_{2}=\{W_{1}\cap W_{2} : W_{1} \in \mathcal{W}_{1},W_{2}\in \mathcal{W}_{2}\}$$ The definition extends naturally to a finite collection of covers. Fix $\mathcal{W}_{1}\in C_{Y}$ and denote by $\mathcal{P}(\mathcal{W}_{1})\in P_{Y}$ the finite partition generated by
$\mathcal{W}_{1}$: that is, if we say $\mathcal{W}_{1} = \{W_{1}^{1} ,\cdots,W^{m}_{1}\}$, $m \in \mathbb{N}$ then $$\mathcal{P}(\mathcal{W}_{1})=
\left\{\cap_{i=1}^{m}A_{i}:A_{i}\in \{W_{1}^{i}, (W_{1}^{i})^{c}\},1\leq i\leq m\right\}.$$

A finite collection of partitions which we will use in the sequel. Let
$$\mathcal{P}(\mathcal{W}_{1}) = \{\alpha\in P_{Y} : P(\mathcal{W}_{1})\succeq \alpha\succeq \mathcal{W}_{1}\}$$
Now let $\mathcal{C}$ be a sub-$\sigma-$algebra of $\mathcal{D}$ and $ \mathcal{W}_{1}\in P_{Y}$. We set
$$H_{\nu}(\mathcal{W}_{1}|\mathcal{C})=-\sum_{W_{1}\in \mathcal{W}_{1}}\int_{Y}\nu(\mathcal{W}_{1}|\mathcal{C})(y)\log\nu(\mathcal{W}_{1}|\mathcal{C})(y)d\nu(y)$$
(by convention, we set $0\log0=0$). Here, $\nu(\mathcal{W}_{1}|\mathcal{C})(y)$ denotes the conditional expectation with respect to $\nu$ of the function $1_{W_{1}}$
relative to $\mathcal{C}$.

Let $(Y,\mathcal{D},\nu,G)$ be an MDS, $\mathcal{W}\in C_{X}$ and $\mathcal{C}\subseteq \mathcal{D}$ a sub-$\sigma-$algebra. For each $F\in Fin(G)$, set
$\mathcal{W}_{F}=\bigvee_{g\in F}g^{-1}\mathcal{W}$.
If $\mathcal{C}$ is $G-$invariant, i.e. $g^{-1}\mathcal{C} = \mathcal{C}$ (up to $\nu$ null sets) for each $g\in G$, let
$H_{\nu}(\mathcal{W}_{\cdot}|\mathcal{C}) :Fin(G)\rightarrow\mathbb{R}$,
$F\mapsto H_{\nu}(\mathcal{W}_{F}|\mathcal{C})$ and the measure-theoretic $\nu-$entropy of $\mathcal{W}$ with respect to $\mathcal{C}$ and the measure-theoretic $\nu,+-$entropy
of $\mathcal{W}$ with respect to $\mathcal{C}$ by
$$h_{\nu}(G,\mathcal{W}|\mathcal{C})=\lim_{n\rightarrow\infty}\frac{1}{|F_{n}|}H_{\nu}(\mathcal{W}_{F_{n}}|\mathcal{C})$$
and
$$h_{\nu,+}(G,\mathcal{W}|\mathcal{C})=inf_{\alpha\in P(Y),\alpha\succeq \mathcal{W}}h_{\nu}(G,\alpha|\mathcal{C})$$

\begin{theorem}(\cite[\textit{Theorem 3.2}]{DZ})\label{a}
Let $(Y, \mathcal{D}, \nu, G)$ be an MDS, $\mathcal{W}\in C_{Y}$ and $\mathcal{C} \subseteq \mathcal{D}$ a $G$-invariant sub-$\sigma$-algebra. Assume that $(Y,\mathcal{D}, \nu)$
is a Lebesgue space. Then
$$h_{\nu}(G, \mathcal{W}|\mathcal{C}) = h_{\nu,+}(G, \mathcal{W}|\mathcal{C}).$$
\end{theorem}

Then the measure-theoretic $\nu-$entropy of $(Y,\mathcal{D},\nu,G)$ with respect to $\mathcal{C}$ is
$$h_{\nu}(G,Y|\mathcal{C})=sup_{\alpha\in P(Y)}h_{\nu}(G,\alpha|\mathcal{C}).$$

In addition, we also need the following theoretial foundation. This part of the content can also be found in the reference \cite[P19]{DZ}. In this article, we use TDS $(Y,G)$
to represent a group $G$ acts over a compact metric space $Y$ as a group of homeomorphisms of the space. We write $P(Y)$ for the set of probability measures on $Y$.

Let $\pi: (Y_{1},G)\rightarrow(Y_{2},G)$ be a factor map between TDSs and $\mathcal{W}\in C_{Y_{1}}, \nu_{1}\in \mathcal{P}(Y_{1},G)$, where denote by $\mathcal{P}(Y_{1},G)$
the set of all $G-$invariant elements of $P(Y_{1})$, is a nonempty compact metric space. so the measure-theoretic $\nu_{1}$-entropy of $\mathcal{W}$ relative to $\pi$ by
$$h_{\nu_{1}}(G,\mathcal{W}|\pi) = h_{\nu_{1}}(G,\mathcal{W}|\pi^{-1}\mathcal{B}_{Y_{2}})= h_{\nu_{1,+}}(G,\mathcal{W}|\pi^{-1}\mathcal{B}_{Y_{2} })$$
where $\mathcal{B}_{Y_{1}},\mathcal{B}_{Y_{2}}$ is the Borel $\sigma-$algebra of $Y_{1},Y_{2}$,respectively. The second equality follows from theorem \ref{a}, since
$(Y_{1}, \mathcal{B}_{Y_{1}}, \nu_{1})$ is a Lebesgue space. Finally, the measure-theoretic $\nu_{1}$-entropy of $(Y_{1},G)$ relative to $\pi$ is
$h_{\nu_{1}}(G, Y_{1}\mid\pi) = h_{\nu_{1}}(G,Y_{1}|\pi^{-1}\mathcal{B}_{Y_{2}})$.

\begin{prop}(\cite[\textit{Lemma 9.5}]{KL17})\label{b}
Let $\mathcal{P}$ be a finite partition of $X$ and let $\varepsilon> 0$. Then there exists a $\delta> 0$ such that, for every finite partition $\mathcal{Q}$ of $X$ with
the property that for all $A\in \mathcal{P} $ there is a set $B$ in the $\sigma-$algebra generated by $\mathcal{Q}$ satisfying $\mu(A\Delta B) < \delta$, one has
$H( \mathcal{P}\mid\mathcal{Q})<\varepsilon$.
\end{prop}

\section{Extension}

In reference \cite{HL}, authors proved the entropy of extended map conditional to the natural extension is zero under the general transformation (see \cite{WP}), and we solve the corresponding result under the action of an Abelian group, since Abelian group is an amenable group and some properties of the amenable group are used in the proof of the main result, some properties of amenable group need to be known, see\cite{HLZ,HLZ2,LHL,ZHL}.
This section is based on the influential articles \cite{AL,KL17,RV}.
By an action of the group $G$ on a set $X$ we mean a map $\alpha:G \times X\longrightarrow X$ such that, writing the first argument as a subscript,
$\alpha_{s}(\alpha_{t}(x)) = \alpha_{st}(x)$ and $\alpha_{e}(x) = x$ for all $x\in X$ and $s, t\in G$. Most of the time we will write the action as $G\curvearrowright X$.

Given two actions $G\curvearrowright X$ and $G\curvearrowright Y$, a map $\varphi:X\rightarrow Y$ is $G-$ equivariant, or simply equivariant, if $\varphi(sx)=s\varphi(x)$ for all $x\in X$ and $s\in G$.

Let $G\curvearrowright X$ and $G\curvearrowright Y$ be continuous actions on compact Hausdorff spaces. We say that the second is a factor of the first, and that the first is an extension of the second, if there is an equivariant continuous surjection $\pi: X \longrightarrow Y $. Such a $\pi$ is called a $G$-factor map or $G$-extension.
A set $A\subseteq X$ is $G-$invariant if $GA=A$, which is equivalent to $GA\subseteq A$.

For a given Lebesgue space $(X,\mathcal{B},\mu)$, $G$ is an Abelian group and $G\curvearrowright X$ is a group action. Let
$$Y=\left\{y\in X^{G}:y(tg)=\alpha_{t^{-1}}y(g) \quad \text{for all}\quad t,g\in G\right\}$$ Then

\begin{enumerate}
\item $Y$ is a $G-$invariant subset of $X^{G}$.

\item $\varphi:Y\rightarrow X$ by $\varphi(y)=y(e)$ is a bijection of measurable space.

\item Taking the $\sigma-$algebra $\mathcal{D}$ on $Y$ which is the restriction of the product $\sigma-$algebra on $X^{G}$. $\mathcal{D}=\mathcal{B}_{X^{G}}\mid_{Y}$,
define the probability measure $\nu$ on $Y$ by  $\nu(A)=\mu(\varphi(A))$ for $A\in \mathcal{D}$. Then $\nu$ is invariant for the restriction of the shift action to $Y$.
\end{enumerate}

\begin{proof}

(1) Let $y\in Y$, $s\in G$.  Consider the left shift action $G\curvearrowright X^{G}$, which for notational legibility in later formulas we will write as
$(sy)(t) = y(s^{-1}t)$ in this section, the property of the elements of $Y$ and the action $G\curvearrowright X$,
$$(sy)(tg)=y(s^{-1}tg)=\alpha_{g^{-1}t^{-1}}(y(s))=\alpha_{g^{-1}}(\alpha_{t^{-1}}y(s^{-1}))=\alpha_{g^{-1}}(y(s^{-1}t))=\alpha_{g^{-1}}((sy)(t))$$ for all $t,g\in G$.
We conclude by the definition of $Y$ and $G$ is Abelian group that $Y$ is a $G-$invariant subset of $X^{G}$.

(2) If $y_{1}\neq y_{2}\in Y$, there is $g\in G$ such that $y_{1}(g)\neq y_{2}(g)$. If $g=e$, it is obvious. If $g\neq e$, then
$y_{1}(g)=y_{1}(e\cdot g)=\alpha_{g^{-1}}(y_{1}(e))$ and $y_{2}(g)=y_{2}(e\cdot g)=\alpha_{g^{-1}}(y_{2}(e))$, and
$$y_{1}(e)=\alpha_{g}\circ \alpha_{g^{-1}}y_{1}(e)=\alpha_{g}y_{1}(g)\neq\alpha_{g}y_{2}(g)=\alpha_{g}\circ \alpha_{g^{-1}}y_{2}(e)=y_{2}(e)$$ i.e. $y_{1}(e)\neq y_{2}(e)$.
Hence  $\varphi$ is injective.

Let $x\in X$, there is $y=(y_{g})_{g\in G}\in Y$ such that $\varphi(y)=x=y(e)$. Let $g\in G$, and define $y(g)=\alpha_{g^{-1}}(y(e))=\alpha_{g^{-1}}(x)$, and so $\varphi$
is surjective.

(3)Let $s\in G$, $G \curvearrowright (X,\mu)$ p.m.p action and $\nu=\varphi(\mu)$, then
$$\nu(s^{-1}A)=\mu(\varphi(s^{-1}A))=\mu(s^{-1}\varphi(A))=\mu(\varphi(A))=\nu(A)   \quad \text{for}\quad   A\in\mathcal{D}.$$
\end{proof}

If $G$ is an Abelian group. Let $g\in G$, $\Pi_{g,X}:Y\rightarrow X$ defined by $\Pi_{g,X}(y)=y(g)$ for $y\in Y$. Then $\prod_{g,X}$ is $G-$factor map for $g\in G$. Indeed, for $g\in G$,
$$\prod_{g,X}\circ \alpha_{s}(y)=(\alpha_{s}(y))(g)=y( s^{-1}(g))=\alpha_{s}(y(g))=\alpha_{s}\circ\prod_{g,X}(y), $$
for all $y\in Y,s\in G$. Hence $\Pi_{g,X}$ is equivariant.

For any $x\in X$, there is $y\in Y$ with $y(s)=s^{-1}gx$ for all $s\in G$ such that $\Pi_{g,X}(y)=y(g)=x$ and so $\Pi_{g,X}$ is surjective. For any open subset $B$ of $X$, $$\Pi_{g,X}^{-1}(B)=\{y\in Y:y(g)\in B\}=Y\cap\{B\times\prod_{s\in G\setminus\{g\}}X\}$$
is an open subset of $Y$ i.e. $\Pi_{g,X}$ is continuous. For a given Lebesgue system $(X,\mathcal{B},\mu,G)$. Let
$\overline{\mathcal{B}}_{g}=\Pi_{g,X}^{-1}(\mathcal{B})$. For $g\in G$. And we get $\Pi_{g,X}^{-1}(\mathcal{B})\subseteq \Pi^{-1}_{e,X}(\mathcal{B})$. This since
$$\Pi_{g,X}^{-1}(B)=\{y\in Y:y(g)\in B\}=\{y\in Y:y(e)=\alpha_{g}(y(g))\in gB\}=\Pi^{-1}_{e,X}(gB)\in \Pi^{-1}_{e,X}(\mathcal{B})$$ for all $B\in \mathcal{B}$.

Set ${\mathcal{D}}_{Y}=\bigcup_{g\in G}\overline{\mathcal{B}}_{g}$, ${\mathcal{D}}_{Y}$ is an algebra of subsets of $Y$. Indeed,

\begin{enumerate}
\item $\emptyset\in {\mathcal{D}}_{Y}$;

\item if $A,B\in {\mathcal{D}}_{Y}$, i.e. there exist $g_{1},g_{2}\in G$ and $B_{1},B_{2}\in \mathcal{B}$ such that $A=\Pi_{g_{1},X}^{-1}(B_{1})$ and
$B=\Pi_{g_{2},X}^{-1}(B_{2})$,then $A\bigcap B=\Pi_{g_{1},X}^{-1}(B_{1})\bigcap\Pi_{g_{2},X}^{-1}(B_{2})=\Pi_{e,X}^{-1}(g_{1}B_{1}\bigcap g_{2}B_{2})\in\mathcal{D}_{Y}$;

\item Let $A\in {\mathcal{D}}_{Y}$, there is $g\in G,B\in\mathcal{B}$ satisfy $A=\Pi_{g,X}^{-1}(B)$ and $Y/A = Y /\Pi_{g,X}^{-1}(B)=\Pi_{g,X}^{-1}(X/B)\in\mathcal{D}_{Y}$.
\end{enumerate}

Define measure $\overline{\mu}$ on ${\mathcal{D}}_{Y}$satisfies $\overline{\mu}(\Pi_{g,X}^{-1}(A))=\mu(A)$ for $A\in \mathcal{B},g\in G$. $\overline{\mathcal{B}}$ is the completion of the $\sigma-$algebra
generated by ${\mathcal{D}}_{Y}$ with respect to $\overline{\mu}$. The self-map $G$ defined on $Y$ by the restriction of the shift action $G\curvearrowright X^{G}$ to $Y$.
The coinduced action is the p.m.p action $G\curvearrowright(Y,\overline{\mu})$.

Let $\Pi_{X}=\Pi_{e,X}$. Then
$\Pi_{X}:(Y,\overline{\mathcal{B}},\overline{\mu},G)\rightarrow(X,\mathcal{B},\mu,G)$ is a factor map and $(Y,\overline{\mathcal{B}},\overline{\mu},G)$ is a natural extension
of $(X,\mathcal{B},\mu,G)$.

\section{Main result}

\begin{theorem}
Let
$\Pi_{X}:(Y,\overline{\mathcal{B}},\overline{\mu},G)\rightarrow(X,\mathcal{B},\mu,G)$ is a natural extension of $(X,\mathcal{B},\mu,G)$. Then $$h_{\overline{\mu}}(G,Y\mid\Pi_{X})=0.$$
\end{theorem}
\begin{proof}
For $g\in G$. Let $\Pi_{g,X}:Y\rightarrow X$ by $\Pi_{g,X}(y)=y(g)$ for $y\in Y$.  Since $\overline{\mathcal{B}}$ is the completion of the $\sigma-$algebra generated by
${\mathcal{D}}_{Y}$ for any $A\in \overline{\mathcal{B}}$ and $\epsilon>0$, there exists $A_{\epsilon}\in{\mathcal{D}}_{Y}$ such that
$\overline{\mu}(A\Delta A_{\epsilon})<\epsilon$ where $A\Delta A_{\epsilon}=(A\backslash A_{\epsilon})\cup (A_{\epsilon}\backslash A)$ ( see \cite[\textit{Theorem 0.7}]{WP}).

Let $\alpha=\{A_{1},A_{2},\cdots,A_{k}\}\in P_{Y}$ with $k\geq 2$. for $m\in \mathbb{N}$, there exists $\delta=\delta(k,m)>0$ such that for every finite partition
$\beta=\{B_{1},B_{2},\cdots,B_{k}\}\in P_{Y}$ with the property that for all $A\in\alpha$ there is a set $B$ in the $\sigma-$algebra generated by $\beta$ satisfying
$\overline{\mu}(A\Delta B)<\delta$, one has $H_{\overline{\mu}}(\alpha\mid\beta)<\frac{1}{m}$ (see Proposition \ref{b}).

For $i=1,2,\cdots,k-1$, we take $A_{i}'\in{\mathcal{D}}_{Y}$ with $\overline{\mu}(A_{i}\Delta A_{i}')<\frac{\delta}{k^{3}}$. Let $A_{k}'=Y\setminus \cup^{k-1}_{j=1}A_{j}'$.
Then $A_{k}'\in{\mathcal{D}}_{Y}$ and
$$\overline{\mu}(A_{k}\Delta A_{k}')\leq\mu(\cup^{k-1}_{j=1}A_{j}\Delta A_{j}')\leq\sum^{k-1}_{j=1}\mu(A_{j}\Delta A_{j}')<\frac{\delta}{k^{2}}$$
where the first inequality comes from the fact that
$$A_{k}\Delta A_{k}'=(X\setminus \bigcup^{k-1}_{j=1}A_{j})\Delta(X\setminus \bigcup^{k-1}_{j=1}A_{j}')\subseteq \bigcup^{k-1}_{j=1}(A_{j}\Delta A_{j}').$$

Then let $C_{1}=A_{1}'$ and
$$C_{2}=A_{2}'\setminus A_{1}'=A_{2}'\setminus C_{1}\in{\mathcal{D}}_{Y},$$
$$C_{3}=A_{3}'\setminus (A_{1}'\cup A_{2}')=A_{3}'\setminus (C_{1}\cup C_{2})\in{\mathcal{D}}_{Y},$$
$$......$$
$$C_{k}=A_{k}'\setminus \bigcup_{j=1}^{k-1}A_{j}'=A_{k}'\setminus \bigcup_{j=1}^{k-1}C_{j}\in {\mathcal{D}}_{Y}$$
Clearly, $C_{i}\in{\mathcal{D}}_{Y}$ for each $i\in\{1,2,\cdots,k\}$ and $\gamma:=\{C_{1},C_{2},\cdots,C_{k}\}\in \mathcal{P}_{Y}$.

For each $i\in\{1,2,\cdots,k\}$ there exist $g\in G$ and $D_{i}\in \mathcal{B}$ such that $\Pi^{-1}_{X}(D_{i})=C_{i}$. Let $\tau=\{D_{1},D_{2},\cdots,D_{k}\}$. Then
$\tau\in\mathcal{P}_{X}$ and $\gamma=\Pi^{-1}_{X}(\tau)$.

For $i\in\{1,2,\cdots,k\}$, we have
\begin{equation*}
\begin{split}
  A_{i}\Delta C_{i}&=(A_{i}\setminus C_{i})\cup(C_{i}\setminus A_{i})\\
  &\subseteq\left( A_{i}\setminus (A_{i}'\setminus \bigcup_{j=1}^{i-1}A_{j}')\right)\bigcup(A_{i}'\setminus A_{i})\\
  &=(A_{i}\setminus A_{i}')\bigcup\left(A_{i}\cap\bigcup_{j=1}^{i-1}A_{j}'\right)\bigcup(A_{i}'\setminus A_{i})\\
  &=(A_{i}\Delta A_{i}')\bigcup\left(\bigcup_{j=1}^{i-1}A_{j}'\cap  A_{i}\right)\\
  &\subseteq(A_{i}\Delta A_{i}')\bigcup\left\{\bigcup_{j=1}^{i-1}\left[(A_{j}\cap  A_{i})\bigcup(A_{i}\cap(A_{j}'\setminus A_{j}))\right]\right\}\\
  &=(A_{i}\Delta A_{i}')\cup\left(\bigcup_{j=1}^{i-1}(A_{i}\cap(A_{j}'\setminus A_{j}))\right)\\
  &\subseteq \bigcup_{j=1}^{i}(A_{j}'\Delta A_{j}).
\end{split}
\end{equation*}

Thus
$$\overline{\mu}( A_{i}\Delta C_{i})<\frac{\delta}{k}<\delta.$$
By the choice of $\delta$, we know $H_{\overline{\mu}}(\alpha\mid\gamma)<\frac{1}{m}$.

Now, let $\{F_{n}\}$ be F{\o}lner sequence, for $n\in \mathbb{N}$, we have

\begin{equation*}
\begin{split}
  H_{\overline{\mu}}(\bigvee_{s\in F_{n}}s^{-1}\alpha\mid \Pi^{-1}_{X}(\mathcal{B}))&\leq H_{\overline{\mu}}(\bigvee_{s\in F_{n}}s^{-1}\alpha\mid\Pi^{-1}_{X}(\bigvee_{s\in F_{n}}s^{-1}\tau))\\
  &=H_{\overline{\mu}}(\bigvee_{s\in F_{n}}s^{-1}\alpha\mid\bigvee_{s\in F_{n}}s^{-1}(\Pi^{-1}_{X}\tau))\\
  &=H_{\overline{\mu}}(\bigvee_{s\in F_{n}}s^{-1}\alpha\mid\bigvee_{s\in F_{n}}s^{-1}\gamma)\\
  &\leq\sum_{s\in F_{n}}H_{\overline{\mu}}(s^{-1}\alpha\mid\bigvee_{s\in F_{n}}s^{-1}\gamma)\\
  &\leq\mid F_{n}\mid H_{\overline{\mu}}(s^{-1}\alpha\mid s^{-1}\gamma)\\
  &=\mid F_{n}\mid H_{\overline{\mu}}(\alpha\mid \gamma).
\end{split}
\end{equation*}

Using the above inequality, we have

\begin{equation*}
\begin{split}
  h_{\overline{\mu}}(G,\alpha\mid\Pi^{-1}_{X}(\mathcal{B}))&=\lim_{n\rightarrow\infty}\frac{1}{\mid F_{n}\mid}H_{\overline{\mu}}(\bigvee_{s\in F_{n}}s^{-1}\alpha\mid \Pi^{-1}_{X}(\mathcal{B}))\\
  &\leq\lim_{n\rightarrow\infty}\frac{1}{\mid F_{n}\mid}\cdot\mid F_{n}\mid H_{\overline{\mu}}(\alpha\mid \gamma)\\
  &<\frac{1}{m}.
\end{split}
\end{equation*}
Since $m$ is arbitrary, $h_{\overline{\mu}}(G,\alpha\mid\Pi^{-1}_{X}(\mathcal{B}))=0$. This implies $h_{\overline{\mu}}(G\mid\Pi_{X})=0$ this since $\alpha$ ia arbitrary. The proof is complete.

\end{proof}



\bibliographystyle{amsplain}

\begin{thebibliography}{10}

\bibitem {AL} L.Arnold,
\textit{Random dynamical systems.Springer Monographs in Mathematics.}
Springer,Berlin,1998.

\bibitem {DZ} A.Dooley, G.Zhang,
\textit{Local entropy theory of a random dynamical system.}
American Mathematical Society, 2015.

\bibitem {HL} W Huang, K. Lu,
\textit{Entropy, Chaos,and Weak Horseshoe for Infinite-Dimensional Random Dynamical Systems.}
Comm. Pure Appl. Math. 2017,70:1987-2036.


\bibitem {HLZ} X. Huang, J. Liu, C.Zhu,
\textit{The Bowen topological entropy of subsets for amenable group actions.}
J Math Anal Appl, 2019, 472(2): 1678-1715.

\bibitem {HLZ2} X.Huang, Y.Lian, C.Zhu,
\textit{A Billingsley-type theorem for the pressure of an action of an amenable group.}
Discrete Contin. Dyn. Syst., 2019,39(2):959-993.

\bibitem {KL17}  D. Kerr,  H. Li,
\textit{Ergodic Theory : Independence and Dichotomies.} Springer,
2016.

\bibitem {LHL} Y.Lian, X.Huang and Z.Li,
\textit{The proximal relation, regionally proximal relation and Banach proximal relation for amenable group actions.}
Acta Mathematica Scientia, 2021,41B(3):729-752.

\bibitem {RV}V.A.Rohlin,
\textit{Lectures on the entropy theory of transformation with invariant measure.}
 UspehiMat.Nauk, 1967,22(5):3-56.

\bibitem {WP} P.Walters,
\textit{An introduction to ergodic theory.} Graduate Texts in
Mathematics, 79. Springer, New York-Berlin, 1982.


\bibitem {ZHL} B.Zhu, X.Huang and Y.Lian,
\textit{The systems with almost Banach mean equicontinuity for Abelian group actions.}
Acta Mathematica Scientia, 2022,42(3):919-940.


\end{thebibliography}

\end{document}